\documentclass[microtype]{gtpart}    

\usepackage{graphicx}  
\usepackage[all]{xy}

\newtheorem{theorem}{Theorem}[section]    
\newtheorem{lemma}[theorem]{Lemma}          
\theoremstyle{definition}
\newtheorem{definition}[theorem]{Definition}    
\newtheorem*{remark}{Remark}             
\theoremstyle{example}
\newtheorem{example}[theorem]{Example}
\newtheorem{proposition}[theorem]{Proposition}

\title{On generalized configuration space and its homotopy groups}

\author{Jun Wang}
\givenname{Jun}
\surname{Wang}
\address{Capital Normal University}
\email{wjun@cnu.edu.cn}
\address{School of Mathematical Science,
         Capital Normal University,
         105 West Third Ring Road North, Haidian District, Beijing, 100048, China}

\author{Xuezhi Zhao}
\givenname{Xuezhi}
\surname{Zhao}
\address{Capital Normal University}
\email{zhaoxve@cnu.edu.cn}
\address{School of Mathematical Science,
         Capital Normal University,
         105 West Third Ring Road North, Haidian District, Beijing, 100048, China}

\begin{document}
\begin{abstract}
 Let $M$ be a subset of vector space or projective space. The authors define the \emph{generalized configuration space} of $M$ which is formed by $n$-tuples of elements of $M$ where any $k$ elements  of each $n$-tuple are linearly independent. The \emph{generalized configuration space}  gives a generalization of the classical configuration space defined by E.Fadell. Denote the \emph{generalized configuration space} of $M$  by $W_{k,n}(M)$. The authors are mainly interested in the calculation about the homotopy groups of  generalized configuration space.  This article gives the fundamental groups of  generalized configuration spaces of  $\mathbb{R}P^m$  for some special cases, and the connections between the homotopy groups of generalized configuration spaces of $S^m$  and the homotopy groups of Stiefel manifolds. It is  also proved that the higher homotopy groups of generalized configuration spaces $W_{k,n}(S^m)$ and $W_{k,n}(\mathbb{R}P^m)$ are isomorphic.

\end{abstract}

\maketitle

\section{Introduction}
In 1962, E.Fadell and L.Neuwirth \cite{Fadell1962} defined configuration space of manifold $M$ which is known as
\[
F_{n}(M)=\{(x_1,...,x_n)\mid x_i\in M, 1\leq i\leq n; x_i\ne x_j, \mbox{ if } i\ne j\}.
\]
After that, Fadell's configuration spaces are studied for decades in the intersections of algebra, geometry and topology (see e.g. \cite{E.R.Fadell;S.Y.Husseinni, Cohen.D.C.2018, GG2017, B.M.P.2017}). Fadell's configuration spaces also attract the interest of other fields such as phylogenetics, robotics and distributed computing (see e.g. \cite{Bjorner2012,Callegaro2016,Ghrist2010} ). \par

In Algebraic topology, there are many works on calculations about the homotopy group, homology group and  cohomology ring of Fadell's configuration spaces where $M$ are different manifolds or algebraic varieties (see e.g. \cite{E.Fadell1962, Benderskly1991,Bodigheimer1989,Totaro1996,C.Schiessl2019}). In 2004, M.Aouina and J.R.Klein \cite{AK2004} proved that a suitable iterated suspension of $F_n(M)$ is a homotopy invariant of PL manifold $M$. In Low-dimensional topology, the pure braid group $PB_n(M)$ defined by E.Artin \cite{Artin1925}\cite{Artin1947}  is  isomorphic to the fundamental group of configuration space $F_{n}(M)$. By this correspondence, configuration spaces form one of the basic tools for studying braids and links such as  finding defining relations in the braid groups of surface (see e.g. \cite{birman}) and finding invariants of knots and links (see e.g. \cite{bott1996}, \cite{bott1998}). There are also some studies about the loop space of configuration space (see e.g. \cite{F.R.Cohen2002}),  the calculation about homotopy groups of sphere (see e.g. \cite{Berrick2005}) and so on. \par

 Morever, the study about  generalization of  Fadell's configuration space attracts many people.  In 1997, M.A.Xicot\'{e}ncatl \cite{Xicotencatl1997} defined \emph{Orbit Configuration spaces} in his Ph.D. Thesis by
\[
Conf^{G}(M,n)=\{(m_1,...m_n)\in M^n\mid G\cdot m_i \cap G\cdot m_j=\varnothing \mbox{ if }i\ne j\}
\]
where the discrete group $G$ acts on a manifold $M$ properly discontinuously. In 2009, F.R.Cohen, T.Kohno and M.A.Xicot\'{e}ncatl \cite{Cohen2009} studied the orbit configuration space associated to $PSL(2,\mathbb{R})$.  In 1998, focusing on complex affine space of $m$-dimension $\mathbb{A}^m$, V.Moulton \cite{Moulton1998} defined the \emph{Vector braids} by the fundamental group of  the configuration space $X_n^m$  which is the space of ordered n-tuples of elements of $\mathbb{A}^m$  with $n\geq m+1$ such that each $m+1$  components of each n-tuple span the whole of $\mathbb{A}^m$ in the sense of affine geometry. In 2015, V.O.Manturov defined a family of groups $G_n^k$ in \cite{Manturov} and  configuration space $C_n^{'}(\mathbb{R}P^k)$ in \cite{Manturov2017} \cite{Manturova} which is the space of ordered $n$-tuples points of $\mathbb{R}P^k$ where no $k$ elements of these points belong to the same projective $(k-2)$-plane,  got that a subgroup of $G_n^k$ is isomorphism to the fundamental group of a subspace of the configuration space $C_n^{'}(\mathbb{R}P^k)$. \par

 In this paper, we generalize Fadell's configuration space by considering $n$-tuples of  elements of the vector space or projective space. Let $M$ be a subset of vector space $V$ over a field $\mathbb{F}$ or projective space $P(V)$ of $V$, we define \emph{generalized configuration space}  $W_{k,n}(M)$   which is formed by $n$-tuples of  elements in $M$ where any $k$  elements of each $n$-tuple are linearly independent (see Section \ref{def}). The aim of this paper is to calculate  homotopy groups of the\emph{ generalized configuration space} of $M$. It is worth mentioning that the generalized configuration spaces of $S^m$ or $\mathbb{R}P^m$ defined in present paper have associations to the Stiefel manifolds. We obtain that if $k=n\leq m+1$, the Stiefel manifold $V_{m+1,n}$ is a deformation retract of generalized configuration space $W_{k,n}(S^m)$. For some $k,n,M$, we also obtain the fiber bundle $\pi: W_{k,n+1}(M)\rightarrow W_{k,n}(M)$ where the map is defined by forgetting the last element of $(n+1)$-tuples in $W_{k,n+1}(M)$. Finally, we obtain the homotopy groups of the generalized configuration spaces of $S^m$ and $\mathbb{R}P^m$ for some special cases. The homotopy groups shall be listed at the end of introduction.\par
   The paper is comprised of three sections besides the introduction. In Section \ref{def}, we define  \emph{generalized configuration space} of $M$  denoted by $W_{k,n}(M)$  and  give some topological properties of the generalized configuration spaces of $M$.  In Section \ref{fiber}, we focus on $\mathbb{F}=\mathbb{R}$. Denote $e_i=(e_i^1,...,e_i^m)\in \mathbb{R}^m$ where $e_i^j=\delta_{i,j} ( =\mbox{ Kronecker delta})$, let $b_k=(e_1,...,e_k)$. For arbitrary point $p\in W_{k,n}(M)$, define $Q_{k-1,n}(M)_p\subset M$ (see Definition \ref{Q}), we obtain the following.
\begin{theorem}
Let the map $\pi:W_{k,n+1}(M)\rightarrow W_{k,n}(M)$  be defined by
\[
\pi(p_1,...,p_n,p_{n+1})=(p_1,...,p_n).
\]
\begin{itemize}
\item (1) Let $M=\mathbb{R}^{m+1} \mbox{ or }S^m$,  if $k=n=m$ or $k=n=m+1$, then
\[
\pi:W_{k,n+1}(M)\rightarrow W_{k,n}(M)
\]
 is a fiber bundle with the fiber $M-Q_{k-1,k}(M)_{b_k}$ and the fiber bundle admits a cross-section.
\item (2) Let $M=\mathbb{R}P^m$, if $k+1=n=m+2$, then $\pi:W_{k,n+1}(\mathbb{R}P^m)\rightarrow W_{k,n}(\mathbb{R}P^m)$ is a fiber bundle  with the fiber $\mathbb{R}P^m-Q_{m,m+2}(\mathbb{R}P^m)_{([e_1],...,[e_{m+1}], [e_1+...+e_{m+1}])}$.
\end{itemize}
\end{theorem}

 The proof of above theorem can be found in Section \ref{fiber}. This theorem is used to calculate   homotopy groups of the generalized configuration spaces $W_{k,n}(M)$. \par
 In Section \ref{homotopy}, let $b_k=(e_1,...,e_k)$, $\tilde{b}_k=(e_1,....,e_{k},\frac{e_1+...+e_{k}}{\|e_1+...+e_{k}\|})$, we obtain the homotopy groups of  $W_{k,n}(S^m)$ and $W_{k,n}(\mathbb{R}P^m)$ for some $k,n$ as follows.
\begin{theorem}
\begin{align*}
\pi_p(W_{k,n}(S^m))&\cong
\begin{cases}
\pi_p(V_{m+1,n}),& \\
                           &1\leq k=n\leq m+1,p\geq 1,\\
  \pi_p(V_{m+1,m})&\oplus \pi_p(S^m-Q_{m-1,m}(S^m)_{b_m})  , \\
                            &1\leq k=n-1=m,p\geq 1.
\end{cases}\\
\pi_p(W_{m+1,m+2}(S^m),\tilde{b}_{m+1})&\cong \pi_p(V_{m+1,m+1},b_{m+1}),p\geq 1.
\end{align*}
\end{theorem}
\begin{theorem}
\begin{align*}
\pi_p(W_{k,n}(\mathbb{R}P^m))&\cong
\pi_p(W_{k,n}(S^m)),p\geq 2;\\
\pi_1(W_{k,n}(\mathbb{R}P^m))&\cong
\begin{cases}
(\mathbb{Z}_2)^n, &k=n\leq m-1,\\
 Q_8, &k=n=m=2,\\
 Q_8\oplus\mathbb{Z}_2, &k=n=m=3,\\
 Q_8\ast_{Z} D_8, &k=n=m=4.
\end{cases}
\end{align*}
\end{theorem}

\section{The definition of generalized configuration spaces $W_{k,n}(M)$}\label{def}
 Let $V$ be a vector space over a field $\mathbb{F}$, $P(V)$ be a projective space of $V$. Let $M$ be a subset of $V$ or $P(V)$, $m$ be the dimension of $M$. Consider $n$-tuples of elements of $M$, define \emph{generalized configuration spaces} of $M$ as follows,
\begin{align*}
W_{k,n}(M)=\{(p_1,...,p_n)\mid & p_i\in M, 1\leq i\leq n; \mbox{ for any k-element subset }\\
 & \{i_1,...,i_{k}\}\subset \{1,...,n\},p_{i_1},...,p_{i_{k}} \mbox{ are linearly independent}\}.
\end{align*}
If $n<k$, it is easily seen that generalized configuration spaces $W_{k,n}(M)$  is equal to $W_{n,n}(M)$, so we will assume throughout that $n\geq k$. Note that  $[v_1],...,[v_k]\in P(V)$ are  linearly independent means that  $v_1,...,v_k\in V$ are linearly independent, it's clear that the definition is well defined.\par
Here are three examples.
\begin{example}
\begin{align*}
W_{1,n}(\mathbb{R}^m)&=\{(p_1,...,p_n)\in (\mathbb{R}^m)^n\mid\forall i\in \{1,...,n\},
p_{i}\mbox{ are linearly independent}\}\\
&=\{(p_1,...,p_n)\in (\mathbb{R}^m)^n\mid\forall i\in \{1,...,n\},
p_{i}\ne 0\}\\
&=(\mathbb{R}^m\setminus\{0\})^n\simeq(S^{m-1})^n.
\end{align*}
\end{example}
\begin{example}
\begin{align*}
&W_{2,n}(S^m)\\
=&\{(p_1,...,p_n)\in (S^m)^n\mid \forall \{i,j\}\subset\{1,...,n\}, p_i,p_j \mbox{ are linearly independent}\}\\
=&\{(p_1,...,p_n)\in (S^m)^n\mid p_i\ne \pm p_j ,\mbox{ if }i\ne j\}.
\end{align*}
It is equal to the spaces $\mathcal{A}_n(S^m)$ defined and studied by E.Fadell in \cite{E.Fadell1962}.
\end{example}
\begin{example}
\begin{align*}
&W_{2,n}(\mathbb{R}P^m)\\
=&\{(p_1,...,p_n)\in(\mathbb{R}P^m)^n\mid\forall \{i,j\}\subset \{1,...,n\},p_{i},p_{j} \mbox{ are linearly independent}\}\\
=&\{(p_1,...,p_n)\in(\mathbb{R}P^m)^n\mid p_{i}\ne p_{j}, \mbox{ if }i\ne j \}.
\end{align*}
It is equal to Fadell's configuration spaces of $\mathbb{R}P^m$: $F_{n}(\mathbb{R}P^m)$.
\end{example}
Denote $M^n$ as the product space of $M$. In general, we can describe the generalized configuration spaces $W_{k,n}(M)$ as $W_{k,n}(M)=M^n-\Delta_{k,n}(M)$  where
\[
\Delta_{k,n}(M)=\bigcup_{\{i_1,...,i_k\}\subset \{1,...,n\}}\{(p_1,...,p_n)\in M^n\mid p_{i_1},...,p_{i_{k}} \mbox{ are linearly dependent}\}.
\]
It is understood that the subspace $\Delta_{k,n}(M)$ of product space $M^n$ is equal to
\[
\Delta_{k,n}(M)=\bigcup_{\{i_1,...,i_k\}\subset \{1,...,n\}}\{(p_1,...,p_n)\in M^n\mid\mbox{rank} (p_{i_1},...,p_{i_{k}})<k\}.
\]
Give generalized configuration spaces $W_{k,n}(M)$ the topology induced by $M$, we obtain the following.
\begin{proposition}
If $M=\mathbb{R}^{k+1},\mathbb{R}P^k$ or $S^k$, then $W_{k,n}(M)$  is a    path-connected
non-compact manifold.
\end{proposition}
\begin{proof}
Let $M=\mathbb{R}^{k+1},\mathbb{R}P^k$ or $S^k$, the point of $W_{k,n}(M)$ can be seen as $n\times (k+1)$-matrix with that any  $n\times k$ matrix formed by $k$ rows of the point has the rank of $k$. It is easily seen that $\Delta_{k,n}(M)$ is determined by polynomial equations which are the  minor matrices of order $k$ of $n\times k$ matrices,  so with the topology induced by $M$, $W_{k,n}(M)$ are manifold. It is convenient to  construct a  path between arbitrary two points by matrix transformation for the proof of path-connected. For non-compactness, construct a sequence of points with the limit $(p_2,p_2,p_3,...,p_n)$, then we obtain the result.
\end{proof}

\begin{proposition}
The generalized configuration space $W_{k,n}(S^m)$ is a deformation retract of  $W_{k,n}(\mathbb{R}^{m+1})$.
\end{proposition}
\begin{proof}
It is convenient to define the deformation retraction of  $W_{k,n}(\mathbb{R}^{m+1})$ onto $W_{k,n}(S^m)$  via  unitization of vectors.
\end{proof}
Recall that the Stiefel manifold \cite{Stiefel1935} is defined by
\begin{align*}
V_{m,n}=\{(v_1,...,v_n)\mid v_k\in \mathbb{R}^m, \|v_k\|=1,1\leq k\leq n;  \langle v_i,v_j\rangle=\delta_{i,j},1\leq i,j\leq n \}.
\end{align*}
 The following Lemma gives the homotopy types of generalized configuration spaces $W_{n,n}(S^m)$.
\begin{lemma}\label{dr}
If $n\leq m+1$, then $V_{m+1,n}$ is a deformation retract of $W_{n,n}(S^{m})$.
\end{lemma}
\begin{proof}
Let $n\leq m+1$, consider the generalized configuration spaces
 \[
 W_{n,n}(S^m)=\{(p_1,...,p_n)\mid p_i\in S^m,1\leq i\leq n; p_1,...,p_n \mbox{ are linearly independent}\}.
 \]
 For any point $p=(p_1,...,p_n)\in W_{n,n}(S^m)$, apply the Gram-Schmidt process, we obtain $p^{'}=(p_1^{'},...,p_n^{'}) \in V_{m+1,n}$.\par
Define the deformation retraction of $W_{n,n}(S^m)$ onto $V_{m+1,n}$ as follows:
 \begin{align*}
 H: W_{n,n}(S^m)\times [0,1]&\rightarrow W_{n,n}(S^m)\\
 (p,t)&\mapsto (\frac{h(p_1,t)}{\|h(p_1,t)\|},\ldots,\frac{h(p_n,t)}{\|h(p_n,t)\|}), t\in[0,1].
 \end{align*}
 where $h(p_i,t)=(1-t)p_i+tp_i^{'}$ for $1\leq i\leq n,t\in [0,1].$
\end{proof}
Now, consider two  subsets of $V$ or $P(V)$ denoted by $M, N$, and a  map $f: M\rightarrow N$.
\begin{definition}
Let $k,n$ be positive integers. The map $f: M\rightarrow N$ is said to be \emph{$(k,n)-$regular} if  for every point $(x_1,...,x_n)\in W_{k,n}(M)$, $(f(x_1),...,f(x_n))\in W_{k,n}(N)$. A map $f: M\rightarrow N$ is said to be \emph{regular} if it is  $(k,n)-$regular for arbitrarily positive integers $k,n$.
\end{definition}
In particular, if $M$ is a covering space over $N$, the following Lemma is an immediate consequence.
\begin{lemma}\label{cover}
Let $P:M\rightarrow N$ be a covering space, $k,n$ positive integers. If both the map $P:M\rightarrow N$ and $P^{-1}: N\rightarrow M$ are $(k,n)-$regular, then $\tilde{P}:W_{n,k}(M)\rightarrow W_{n,k}(N)$ defined by
\[
\tilde{P}(x_1,...,x_n)= (P(x_1),...,P(x_n))
\]
is a covering space.
\end{lemma}
\begin{proof}
 It is understood that if both $P$ and $P^{-1}$ are $(k,n)-$regular, then $\forall (x_1,...,x_n)\in W_{k,n}(M)$,
  \[
  (P(x_1),...,P(x_n))\in W_{k,n}(N)
  \]
  and $\forall (y_1,...,y_n)\in W_{k,n}(N)$,
  \[
  (P^{-1}(y_1),...,P^{-1}(y_n))\in W_{k,n}(M),
  \]
  thus we obtain
\[
\tilde{P}^{-1}(W_{k,n}(N))=W_{k,n}(M).
\]
It is known that if $P:M\rightarrow N$ be a covering space,  then $\tilde{P}:M^n\rightarrow N^n$ is a covering space, thus $\tilde{P}:W_{k,n}(M)\rightarrow W_{k,n}(N)$ is a covering space.
\end{proof}
\begin{example}
Let $P: S^m\rightarrow \mathbb{R}P^m$ be the covering space induced by action of $\mathbb{Z}_2$, then for arbitrarily positive integers $k,n$, $\tilde{P}:W_{k,n}(S^m)\rightarrow W_{k,n}(\mathbb{R}P^m)$ is a $2^n$-sheeted covering space. The group of deck transformations of this covering space is the direct sum of $n$ copies of $\mathbb{Z}_2$ which is denoted by   $(\mathbb{Z}_2)^n$.
\end{example}
\section{The Fiber bundle structure}\label{fiber}
In this section, we shall prove that   $\pi: W_{k,n+1}(M)\rightarrow W_{k,n}(M)$ defined by
\[
\pi(p_1,...,p_n,p_{n+1})=(p_1,...,p_{n})
\]
is a fiber bundle for some $k,n,M$. \par
In particular, if $k=1$, then $W_{1,n}(M)$ is the product space of $M-\{0\}$, thus the map $\pi:W_{1,n+1}(M)\rightarrow W_{1,n}(M)$ defined by
\begin{align*}
\pi(p_1,...,p_n,p_{n+1})= (p_1,...,p_n)
\end{align*}
is a trivial fibration over $ W_{1,n}(M)$ with the fiber $M-\{0\}$.

Let $k=2$, E.Fadell proved the following:
\begin{lemma}\cite[Theorem 2.2]{E.Fadell1962}
The map $\pi:W_{2,n}(S^m)\rightarrow W_{2,n-1}(S^m),n\geq 2$ defined by
\[
\pi(p_1,...,p_n)=(p_2,...,p_n)
\]
is a locally trivial fiber map with fiber $S^m-M_{2(n-1)}$ where $M_{2(n-1)}$ is a fixed set of $2(n-1)$ distinct points of $S^m$.
\end{lemma}
For other cases, first, define subspaces of $M$  as follows.
\begin{definition}\label{Q}
For $p=(p_1,...,p_n)\in W_{k,n}(M)$, $ \{i_1,...,i_{k-1}\}\subset \{1,...,{n}\}$, define that
\begin{align*}
Q_{n}^{i_1,...,i_{k-1}}(M)_p=\{x\in M\mid x,p_{i_1},...,p_{i_{k-1}} \mbox{ are linearly dependent}\}.
\end{align*}
Denote that $Q_{k-1,n}(M)_p=\bigcup_{\{i_1,...,i_{k-1}\}\subset\{1,...,{n}\}}Q_{n}^{i_1,...,i_{k-1}}(M)_p$.
\end{definition}
In particular, if $k=1$, we have $Q_{k-1,n}(M)_p=\{0\}\subset M$. Now, focus on the generalized configuration spaces $W_{k,k}(S^m), m\geq 1$,  let $b_k=(e_1,...,e_{k})$ be the base point of $W_{k,k}(S^m)$.
\begin{theorem}\label{bundle}
Let the map $\pi:W_{k,k+1}(S^m)\rightarrow W_{k,k}(S^m)$  be defined by
\begin{align*}
\pi(p_1,...,p_{k},p_{k+1})=(p_1,...,p_{k}).
\end{align*}
If $k=m$ or $m+1$, then  $\pi: W_{k,k+1}(S^m)\rightarrow W_{k,k}(S^m)$ is a fiber bundle with the fiber $S^m-Q_{k-1,k}(S^m)_{b_k}$.\par
\end{theorem}
\begin{proof}
Let $k=m$, consider the generalized configuration spaces
\begin{align*}
W_{m,m}(S^m)=&\{(p_1,...,p_m)\in (S^m)^m\mid p_1,\ldots,p_m\mbox{ are linearly independent}\},\\
W_{m,m+1}(S^m)=&\{(p_1,...,p_{m+1})\in (S^m)^{m+1}\mid \forall \{i_1,...,i_m\}\subset \{1,\ldots,m+1\},\\
                                               &p_{i_1},...,p_{i_m}\mbox{ are linearly independent}\}.
\end{align*}\par
First, for arbitrary point $ q=(q_1,...,q_{m})\in W_{m,m}(S^m)$,
\[
\pi^{-1}(q)=\{(q_1,...,q_{m},x)\mid\mbox{any }m\mbox{ points of the set }\{q_1,...,q_{m},x\}\mbox{ are linearly independent}\},
\]
there is a homeomorphism
\begin{align*}
\pi^{-1}(q)&\rightarrow Q_{m-1,m}(S^m)_q\\
(q_1,...,q_{m},x)&\mapsto (x)
\end{align*}
Now, define the local trivialization of the fiber bundle. Choose the neighbourhood of $q$ in $W_{m,m}(S^m)$ denoted by $U_{q}$ such that $U_{q}$ is small enough. It is understood that for  $x=(x_1,...,x_{m})\in U_q$, $x_1,...,x_m$ are linearly independent. \par
Take $x_{m+1}\in$ Span$\{x_1,...,x_{m}\}^{\perp}\subset\mathbb{R}^{m+1}$ such that
\[
\|x_{m+1}\|=1, \det(x_1,...,x_{m},x_{m+1})>0.
\]
The choice   of  $x_{m+1}$ makes that $x_{m+1}$ only depends on $x=(x_1,...,x_{m})$. Take the point $e_{m+1}=(0,...,0,1)\in \mathbb{R}^{m+1}$. It's easily seen that $\{x_1,...,x_{m+1}\}$ and  $\{e_1,...,e_{m+1}\}$ both are basis of $\mathbb{R}^{m+1}$, then there is a linear transformation $\varphi_{x}:\mathbb{R}^{m+1}\rightarrow \mathbb{R}^{m+1}$ defined by
\[
(\varphi_{x}(e_1),...,\varphi_{x}(e_{m}),\varphi_{x}(e_{m+1}))=(x_1,...,x_{m},x_{m+1}).
\]
Define the map $\psi:\mathbb{R}^{m+1}\rightarrow \mathbb{R}^{m+1}$ by $z\mapsto \frac{z}{\|z\|}$ and  we obtain the map $f_x:S^m\rightarrow S^m$ defined by
\begin{align*}
f_{x}=\psi\circ\varphi_{x}\mid_{S^{m}}:S^{m}&\rightarrow S^{m}\\
\lambda_1e_1+\ldots+\lambda_{m}e_{m}+\lambda_{m+1}e_{m+1}&\mapsto \frac{\lambda_1x_1+\ldots+\lambda_{m}x_{m}+\lambda_{m+1}x_{m+1}}{\| \lambda_1x_1+\ldots+\lambda_{m}x_{m}+\lambda_{m+1}x_{m+1}\|}
\end{align*}
where $\lambda_i\in \mathbb{R}, 1\leq i\leq {m+1}$.  Recall the definition of  the space $Q_{m-1,m}(S^{m})_x$ and $Q_{m-1,m}(S^{m})_{b_m}$,
 \begin{align*}
 &Q_{m-1,m}(S^{m})_x\\
 =&\bigcup_{\{i_1,...,i_{m-1}\}\subset\{1,...,{m}\}}Q_{m}^{i_1,...,i_{m-1}}(S^{m})_x\\
          =&\bigcup_{\{i_1,...,i_{m-1}\}\subset\{1,...,{m}\}}\big(\mbox{Span}\{x_{i_1},...,x_{i_{m-1}}\}\cap S^{m}\big)\\
          =&\bigcup_{\{i_1,...,i_{m-1}\}\subset\{1,...,{m}\}}\bigg(\{\lambda_{i_1} x_{i_1}+\ldots+\lambda_{i_{m-1}} x_{i_{m-1}}\mid\lambda_{i_1},...,\lambda_{i_{m-1}}\in \mathbb{R}\}\cap S^{m}\bigg),\\
 &Q_{m-1,m}(S^{m})_{b_m}\\
 =  &\bigcup_{\{i_1,...,i_{m-1}\}\subset\{1,...,{m}\}}\bigg(\{\lambda_{i_1} e_{i_1}+\ldots+\lambda_{i_{m-1}} e_{i_{m-1}}\mid\lambda_{i_1},...,\lambda_{i_{m-1}}\in \mathbb{R}\} \cap S^{m}\bigg),
 \end{align*}
then we obtain that
\[
f_{x}\mid_{S^m-Q_{m-1,m}(S^{m})_{b_m}}:S^m-Q_{m-1,m}(S^{m})_{b_m}\rightarrow S^m-Q_{m-1,m}(S^{m})_x
\]
 is a homeomorphism.\par
Define the local trivialization of the fiber bundle as follows:
\begin{align*}
h:\pi^{-1}(U_q)&\rightarrow U_{q}\times (S^m-Q_{m-1,m}(S^{m})_{b_m})\\
(x_1,...,x_{m},y)&\mapsto (x_1,...,x_{m},f^{-1}_{x}(y))
\end{align*}
and  $h^{-1}(x_1,...,x_{m},y^{'})=(x_1,...,x_{m},f_{x}(y^{'}))$.\par
It's not hard to verify  the following diagram commutative.
\[
\xymatrix{
 \pi^{-1}(U_q) \ar[d]_{\pi} \ar[r]^-{h} & U_{q}\times (S^m-Q_{m-1,m}(S^{m})_{b_m}) \ar[dl]^{\mbox{projection}}      \\
  U_q                      }
                \]
Then the proof for the case $k=m$ is completed. For the case $k=m+1$, the proof is similar.
\end{proof}

\begin{theorem}\label{c-s}
Let $k=m$ or $m+1$. The fiber bundle
 \[
 S^m-Q_{k-1,k}(S^{m})_{b_k} \hookrightarrow W_{k,k+1}(S^m)\xrightarrow{\pi} W_{k,k}(S^m)
 \]
admits a cross-section.
\end{theorem}

\begin{proof}
By Theorem \ref{bundle}, if $k=m$ or $m+1$, then we obtain the fiber bundle.  Define a map as follows:
\begin{align*}
f:W_{k,k}(S^m)&\rightarrow W_{k,k+1}(S^m)\\
(p_1,...,p_{k})&\mapsto (p_1,...,p_{k},\frac{p_1+\ldots+p_{k}}{\|p_1+\ldots+p_{k}\|}).
\end{align*}
It is understood  that  $p_1,...,p_{k}$ are linear independent, then any $k$ vectors  of the set $\{p_1,...,p_{k},\frac{p_1+\ldots+p_{k}}{\|p_1+\ldots+p_{k}\|}\}$ are linear independent, thus
\[
 (p_1,...,p_{k},\frac{p_1+\ldots+p_{k}}{\|p_1+\ldots+p_{k}\|})\in W_{k,k+1}(S^m).
  \]
  It is easily seen that $\pi\circ f=id_{W_{k,k}(S^{k})}$, then it's a cross-section of the fiber bundle.
\end{proof}
\begin{remark}\label{cross}
For arbitrary point $p=(p_1,...,p_m)\in W_{m,m}(S^m)$ where $p_i$ is denoted by $p_i=(p_i^1,...,p_i^{m+1})\in S^m,1\leq i\leq m$, define the  matrices
\[
D^{\hat{j}}_p=
\left(
  \begin{array}{ccc}
    p_{1}^1& \ldots & p_m^1 \\
    \vdots & \vdots & \vdots \\
    p_{1}^{j-1}& \ldots & p_m^{j-1} \\
    p_1^{j+1}& \ldots & p_m^{j+1}\\
    \vdots & \vdots & \vdots\\
   p_1^{m+1}& \ldots & p_m^{m+1}\\
  \end{array}
\right), 1\leq j\leq m+1.
\]
We obtain that the fiber bundle
\[
S^m-Q_{m-1,m}(S^{m})_{b_m} \hookrightarrow W_{m,m+1}(S^m)\xrightarrow{\pi} W_{m,m}(S^m)
\]
admits another cross-section  $\mathcal{K}$ defined as follows.
\begin{align*}
\mathcal{K}: W_{m,m}(S^m)&\rightarrow W_{m,m+1}(S^m)\\
(p_1,...,p_m) &\mapsto (p_1,...,p_m, \frac{y}{  \parallel y \parallel })
\end{align*}
where $y= ((-1)^{1+m+1}\det D^{\hat{1}}_p,...,(-1)^{j+m+1}\det D^{\hat{j}}_p,...,(-1)^{m+m+1}\det D^{\hat{m}}_p)$ which means
\[
y\in \mbox{Span}\{p_1,...,p_m\}^{\perp}\subset \mathbb{R}^{m+1}, \det (p_1,...,p_m,y)>0.
\]
It is not hard to verify that the map $\mathcal{K}$ is a cross-section.
\end{remark}
Let $M=\mathbb{R}^{m+1},\mathbb{R}P^m$, the similar argument as in the proof of Theorem \ref{bundle} show that the following theorems.
\begin{theorem}
 Let the map $\pi:W_{k,k+1}(\mathbb{R}^{m+1})\rightarrow W_{k,k}(\mathbb{R}^{m+1})$ be defined by
\begin{align*}
\pi(p_1,...,p_{k},p_{k+1})=(p_1,...,p_{k}).
\end{align*}
If $k=m$ or $m+1$, then  $\pi:W_{k,k+1}(\mathbb{R}^{m+1})\rightarrow W_{k,k}(\mathbb{R}^{m+1})$ is a fiber bundle  with the fiber $\mathbb{R}^{m+1}-Q_{k-1,k}(\mathbb{R}^{m+1})_{b_k}$ and the fiber bundle admits a cross-section.
\end{theorem}
\begin{theorem}
 Let the map $\pi:W_{m+1,m+3}(\mathbb{R}P^m)\rightarrow W_{m+1,m+2}(\mathbb{R}P^m)$ be defined by
\begin{align*}
\pi(p_1,...,p_{m+2},p_{m+3})=(p_1,...,p_{m+2})
\end{align*}
Then $\pi:W_{m+1,m+3}(\mathbb{R}P^m)\rightarrow W_{m+1,m+2}(\mathbb{R}P^m)$ is a fiber bundle  with the fiber $\mathbb{R}P^m-Q_{m,m+2}(\mathbb{R}P^m)_{([e_1],...,[e_{m+1}], [e_1+...+e_{m+1}])}$.
\end{theorem}
\begin{proof}
For arbitrary point $p=(p_1,...,p_{m+2})\in W_{m+1,m+2}(\mathbb{R}P^m)$, it is known that $p_1,...,p_{m+2}\in \mathbb{R}P^m$ are on the general position, then we obtain the local trivialization of the fiber bundle via the projective transformation.
\end{proof}

\section{The calculation of Homotopy groups}\label{homotopy}
In this section, we shall  calculate  homotopy groups of  generalized configuration spaces $W_{k,n}(M)$ for some special cases. If $k=1$, then $W_{1,n}(M)$ is the product space of $M-\{0\}$, we obtain that
\[
\pi_pW_{1,n}(M)\cong  (\pi_p(M-\{0\}))^n, p\geq 1,
\]
so in this section, we will assume throughout that $k\geq 2$.\par
\subsection{ Case 1: $M=S^m$}
Let $M=S^m$, the following lemma is a direct result by Lemma \ref{dr}.
\begin{lemma}\label{k}
$\pi_p(W_{n,n}(S^m))\cong \pi_p (V_{m+1,n}), n\leq m+1, p\geq 1.$
\end{lemma}
First, consider the case $k=m$, note that $m \geq 2$, let $b_m=(e_1,...,e_m)$ be the base point of $W_{m,m}(S^m)$, $b_{m+1}=(e_1,...,e_{m+1})$ be the base point of $W_{m,m+1}(S^m)$ and  $e_{m+1}$ be the base point of $S^m-Q_{m-1,m}(S^m)_{b_m}$. The following theorem shows the homotopy groups of  $W_{m,m+1}(S^m)$.
\begin{theorem}\label{m}
\[
\pi_p(W_{m,m+1}(S^m))\cong  \pi_p(V_{m+1,m}) \oplus  \pi_p(S^m-Q_{m-1,m}(S^m)_{b_m}),p\geq 1.
\]
\end{theorem}
\begin{proof}
 By Theorem \ref{bundle}, the fiber bundle
\[
S^m-Q_{m-1,m}(S^m)_{b_m} \hookrightarrow W_{m,m+1}(S^{m}) \xrightarrow{\pi} W_{m,m}(S^{m})
\]
induces a long exact sequence
\begin{align*}
\cdots\rightarrow \pi_p(S^m-Q_{m-1,m}(S^m)_{b_m})\xrightarrow{i_*} \pi_p(W_{m,m+1}(S^{m}))\xrightarrow{\pi_*} \pi_p(W_{m,m}(S^m))\rightarrow\cdots
\end{align*}
where $i_*$ is induced by the inclusion map
\begin{align*}
i:S^m-Q_{m-1,m}(S^m)_{b_m} &\rightarrow W_{m,m+1}(S^{m})\\
          x&\mapsto (e_1,...,e_m,x)
\end{align*}
and $\pi_*$ is induced by the projection map
\begin{align*}
\pi:W_{m,m+1}(S^m)&\rightarrow W_{m,m}(S^m)\\
(x_1,...,x_m,x_m+1)&\mapsto (x_1,...,x_m).
\end{align*}

By Theorem \ref{c-s}, the fiber bundle admits a cross-section, following Sze-Tsen Hu \cite[Proposition V6.2]{hu}, we  obtain that the following  exact sequence is spilt for $p\geq 1$.
\begin{align*}
1\rightarrow \pi_{p}(S^m-Q_{m-1,m}(S^m)_{b_m},e_{m+1})\xrightarrow {i_*} \pi_{p}(W_{m,m+1}&(S^m),b_{m+1})\\
                                                                                       &\xrightarrow{\pi_*} \pi_{p}(W_{m,m}(S^m),b_m)\rightarrow 1.
\end{align*}
Then if $p\geq 2$, by the split exact sequence, we obtain
\[
\pi_{p}(W_{m,m+1}(S^m))\cong \pi_{p}(W_{m,m}(S^m))\oplus \pi_{p}(Q_{m-1,m}(S^m)_{b_m})), p\geq 2.
\]\par
If $p=1$,  then $\pi_{1}(W_{m,m+1}(S^m),b_m)$ is isomorphic to the semi-direct product of $\pi_{1}(S^m-Q_{m-1,m}(S^m)_{b_m},e_{m+1})$ and $\pi_{1}(W_{m,m}(S^m),b_m)$ corresponding the homomorphism
\[
\eta: \pi_{1}(W_{m,m}(S^m),b_m)\rightarrow \mbox{ Aut}(\pi_{1}(S^m-Q_{m-1,m}(S^m)_{b_m}),e_{m+1}).
\]
To complete the proof, it is necessary to prove that the semi-direct product is a direct product. \par
By Lemma \ref{k}, we have $\pi_1(W_{m,m}(S^m),b_m)\cong \pi_1(V_{m+1,m},b_m)\cong \mathbb{Z}_2, m\geq 2$, then denote  the generator of $\pi_{1}(W_{m,m}(S^m),(e_1,..,e_m))$ by $\alpha$ which is given by
\[
\alpha(t)=(r(t)e_1,...,r(t)e_m)
\]
where $r(t)$ is defined by
\[
r(t)=
\left(
  \begin{array}{ccccccc}
      I_{m-1} &0              &0\\
      0       & \cos 2\pi t &\sin 2\pi t \\
       0            &-\sin 2\pi t &\cos 2\pi t
  \end{array}
\right), 0\leq t\leq 1.
\]
Here $I_{m-1}$ is the $(m-1)$-rowed identity matrix. Denote the elements of the fundamental group $\pi_{1}(S^m-Q_{m-1,m}(S^m)_{b_m},e_{m+1})$ by $x$ which is given by $x(t)$ where
\[
x(0)=x(1)=e_{m+1}, x(t)\in S^m-Q_{m-1,m}(S^m)_{b_m}.
\]
Consider the homomorphism
\[
\mathcal{K}_{*}:\pi_1(W_{m,m}(S^m),(e_1,...,e_m))\rightarrow\pi_1(W_{m,m+1}(S^m),(e_1,...,e_m,e_{m+1}))
\]
induced by the cross-section $\mathcal{K}$  defined in Remark \ref{cross}. Hence to show that  the semi-direct product is a direct product, it is only necessary to  prove $x^{\eta(\alpha)}=x$ which means
\[
\mathcal{K}_*(\alpha) \circ i_*(x) \circ \mathcal{K}_*(\alpha^{-1})\sim i_*(x).
\]
By the definition of  cross-section $\mathcal{K}$ in Remark \ref{cross} and the inclusion map, we obtain
\[
\mathcal{K}_*(\alpha(t))=\mathcal{K}_*(e_1,...,e_{m-1},r(t)e_m)=(e_1,...,e_{m-1},r(t)e_m,r(t)e_{m+1})
\]
and the image of $x$ under the  homomorphism induced by the inclusion map is
\[
i_*(x(t))=(e_1,...,e_m,x(t)).
\]
It is understood that  the product path
\[
 \mathcal{K}_*(\alpha) \circ i_*(x) \circ \mathcal{K}_*(\alpha^{-1}) =
\begin{cases}
(e_1,...,e_{m-1},r(3t)e_m,r(3t)e_{m+1}),  &0\leq t\leq \frac{1}{3},\\
(e_1,...,e_{m},x(3t-1)),    &\frac{1}{3}\leq t\leq \frac{2}{3},\\
 (e_1,...,e_{m-1},r(3-3t)e_m,r(3-3t)e_{m+1}),   &\frac{2}{3}\leq t\leq 1.
\end{cases}
\]
Define the homotopy $H:[0,1]\times [0,1] \rightarrow \pi_1(W_{m,m+1}(S^m),b_{m+1})$ by
\[
H(s,t)=
\begin{cases}
(e_1,...,e_{m-1},r(3t(1-s))e_m,r(3t(1-s))e_{m+1}),   &0\leq t\leq \frac{1}{3},\\
(e_1,...,e_{m-1},r(1-s)e_m,r(1-s)x(3t-1)),    &\frac{1}{3}\leq t\leq \frac{2}{3},\\
 (e_1,...,e_{m-1},r((3-3t)(1-s))e_m,r((3-3t)(1-s))e_{m+1}),  &\frac{2}{3}\leq t\leq 1.
\end{cases}
\]
It's not hard to verify that $H(s,t)\in \pi_1(W_{m,m+1}(S^m),b_{m+1})$. Because that $H(1,t)=i_*(x(t))$, we obtain
\[
\mathcal{K}_*(\alpha) \circ i_*(x) \circ \mathcal{K}_*(\alpha^{-1})\sim i_*(x),
\]
then the semi-direct product is the direct product of the groups $\pi_1(W_{m,m}(S^m))$ and $\pi_{1}(S^m-Q_{m-1,m}(S^m)_{b_m})$.\\
Thus we obtain
\[
\pi_p(W_{m,m+1}(S^m))\cong  \pi_p(W_{m,m}(S^m)) \oplus  \pi_p(S^m-Q_{m-1,m}(S^m)_{b_m}), p\geq 1.
\]
By Lemma \ref{k}, we obtain that $\pi_p(W_{m,m}(S^m))\cong \pi_p(V_{m+1,m}), p\geq 1$, then
\begin{align*}
\pi_p(W_{m,m+1}(S^m))\cong  \pi_p(V_{m+1,m}) \oplus  \pi_p(S^m-Q_{m-1,m}(S^m)_{b_m}), p\geq 1.
\end{align*}
\end{proof}
As an example let $k=m=3$.  The homotopy groups of fiber $S^3-Q_{2,3}(S^3)_{b_3}$ are as follows.
\begin{lemma}\label{tree}

\[
\pi_p(S^3-Q_{2,3}(S^3)_{b_3})\cong
\begin{cases}
 \ast_7\mathbb{Z},& p=1,\\
 1,                &p\geq 2.
\end{cases}
\]

\end{lemma}
\begin{proof}
Let $\{i,j\}\subset \{1,2,3\}$, recall that
\begin{align*}
Q_{3}^{i,j}(S^3)_e=\{x\in S^3\mid x,e_i,e_j \mbox{ are linearly dependent}\}=\mbox{Span}\{e_i,e_j\}\cap S^3
\end{align*}
where $\mbox{Span}\{e_i,e_j\}$ is a 2-dimension subspace of $\mathbb{R}^4$. Thus $Q_{3}^{i,j}(S^3)_e$ is the unit circle in $\mbox{Span}\{e_i,e_j\}$  and $Q_{2,3}(S^3)_e=Q_{3}^{1,2}(S^3)_e\cup Q_{3}^{1,3}(S^3)_e\cup Q_{3}^{2,3}(S^3)_e$ is the union of three unit circles in Span$\{e_1,e_2,e_3\}$ (see Figure \ref{e}).

\begin{figure}[ht]
\centering
\includegraphics[scale=0.3]{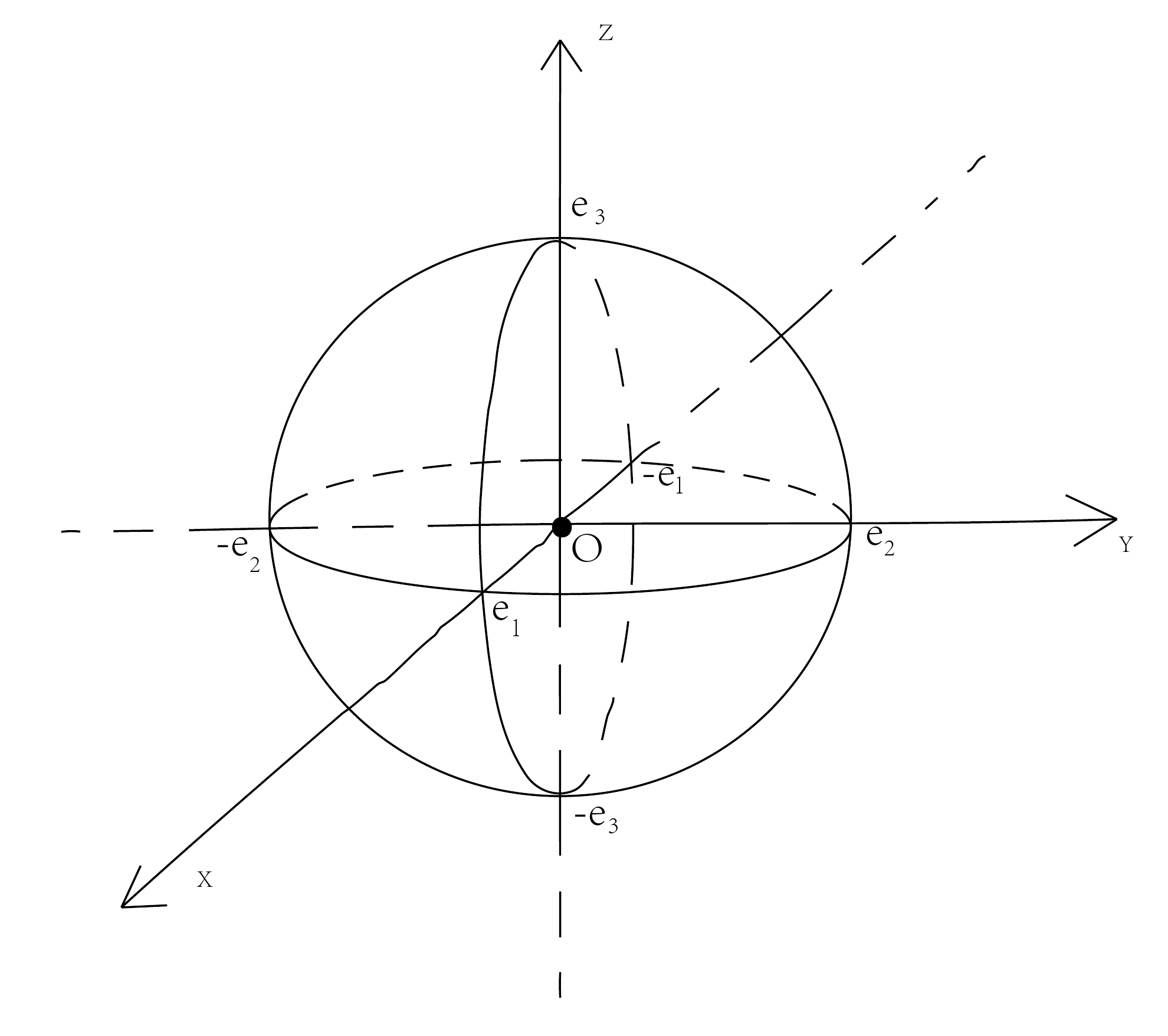}
\caption{$Q_{2,3}(S^3)_{b_3}\subset\mbox{Span}\{e_1,e_2,e_3\}$}
\label{e}
\end{figure}

Actually,  $Q_{2,3}(S^3)_{b_3}$ is a graph in $S^3$ which has 6 vertexes and 12 edges, its complementary space $S^3-Q_{2,3}(S^3)_{b_3}$ is  homotopy equivalent to the dual graph (see Figure \ref{dual}) which has two vertexes and eight edges.

\begin{figure}[ht]
\centering
\includegraphics[scale=0.3]{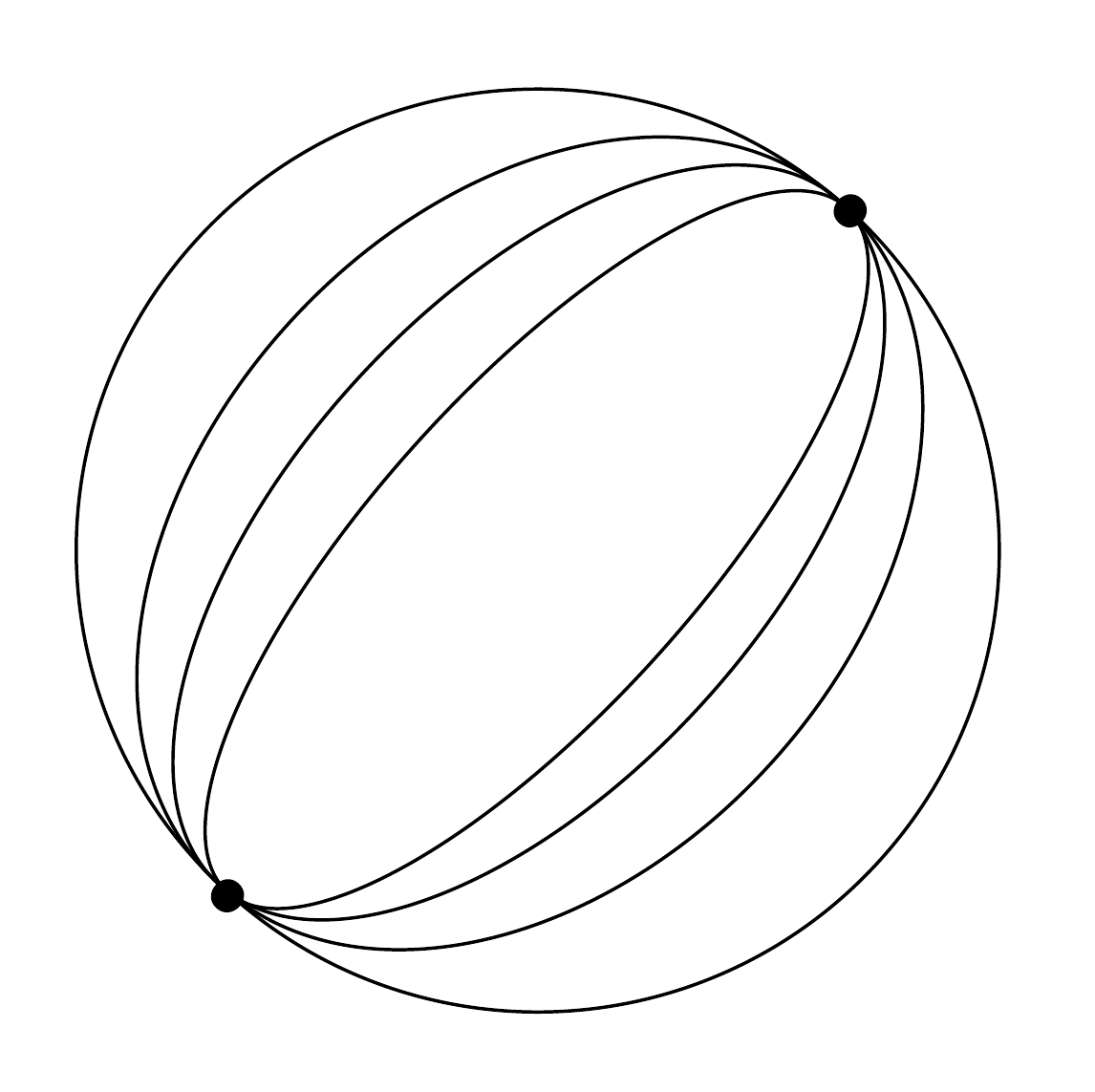}
\caption{The dual graph}
\label{dual}
\end{figure}

Thus  the fundamental group of $S^3-Q_{2,3}(S^3)_{b_3}$ is isomorphic to the free product of seven copies of $\mathbb{Z}$ : $\mathbb{Z}\ast\mathbb{Z}\ast\mathbb{Z}\ast\mathbb{Z}\ast\mathbb{Z}\ast\mathbb{Z}\ast\mathbb{Z}$ denoted by $\ast_{7}\mathbb{Z}$, and
\[
\pi_p(S^3-Q_{2,3}(S^3)_{b_3})=1, p\geq 2.
\]
\end{proof}

By Lemma \ref{k}, we obtain
\[
\pi_p(W_{3,3}(S^3))\cong \pi_p(V_{4,3})\cong
\begin{cases}
 \mathbb{Z}_2,& p=1,\\
 1,&p=2,\\
  \pi_p(S^2)\oplus\pi_p(S^2),                &p\geq 3.
\end{cases}
\]
then the following theorem is a direct result of  Theorem \ref{m} and Lemma \ref{tree}.
\begin{theorem}
\[
\pi_p(W_{3,4}(S^3))\cong
\begin{cases}
 (\ast_7\mathbb{Z})\oplus\mathbb{Z}_2 ,&p=1,\\
 1,&p=2,\\
 \pi_p(S^2)\oplus \pi_p(S^2), &p\geq 3.
\end{cases}
\]
\end{theorem}

Now, consider the second case $k=m+1$, note that $m\geq 1$, let $b_{m+1}$ be the base point of generalized configuration space $W_{m,m+1}(S^m)$ and Stiefel manifold $V_{m+1,m+1}$, $\tilde{b}_{m+1}=(e_1,....,e_{m+1},\frac{e_1+...+e_{m+1}}{\|e_1+...+e_{m+1}\|})$  be the base point of  generalized configuration space $W_{m+1,m+2}(S^m)$. The following theorem shows the homotopy groups of  generalized configuration space $W_{m+1,m+2}(S^m)$.
\begin{theorem}
$\pi_p(W_{m+1,m+2}(S^m),\tilde{b}_{m+1})\cong \pi_p(V_{m+1,m+1},b_{m+1}),p\geq 1$
\end{theorem}
\begin{proof}
By similar arguments as in the proof of Theorem \ref{m}, we obtain that the following exact sequence is split for $p\geq 1$.
\begin{align*}
1\rightarrow \pi_{p}(S^m-Q_{m,m+1}(S^m)_{b_{m+1}}, \tilde{e}_{m+1})\rightarrow  \pi_{p}&(W_{m+1,m+2}(S^m),\tilde{b}_{m+1})\\
&\xrightarrow{\pi_*} \pi_{p}(W_{m+1,m+1}(S^m),b_{m+1})\rightarrow 1.
\end{align*}
where $\tilde{e}_{m+1}=\frac{e_1+...+e_{m+1}}{\|e_1+...+e_{m+1}\|}$ is the base point of the space $S^m-Q_{m,m+1}(S^m)_{b_{m+1}}$. It's clearly that $Q_{m,m+1}(S^m)_{b_{m+1}}\subset \mathbb{R}^{m+1}$ is a union of $(m-1)$-dimension unit sphere, then we obtain
$S^m-Q_{m,m+1}(S^m)_{b_{m+1}} \simeq \bigsqcup_{j=1}^{2^{m+1}} D_j^{m}$ where $D_j^m,1\leq j\leq 2^{m+1}$ are  $m$-dimension disks, then
\[
 \pi_{p}(W_{m+1,m+2}(S^m),\tilde{b}_{m+1})\cong \pi_{p}(W_{m+1,m+1}(S^m),b_{m+1}),p\geq 1,
\]
thus $\pi_{p}(W_{m+1,m+2}(S^m),\tilde{b}_{m+1})\cong \pi_{p}(V_{m+1,m+1},b_{m+1}),p\geq 1$ by Lemma \ref{k}.
\end{proof}

\subsection{Case 2: $M=\mathbb{R}P^m$}
Let $M=\mathbb{R}P^m$, by Theorem \ref{cover}, $\tilde{P}:W_{k,n}(S^m)\rightarrow W_{k,n}(\mathbb{R}P^m)$ is a $2^n$-sheeted covering space and the group of deck transformations of this covering space  is $(\mathbb{Z}_2)^n$, then we obtain the following lemma.
\begin{lemma}
$\pi_p(W_{k,n}(\mathbb{R}P^m))\cong \pi_p(W_{k,n}(S^m)),p\geq 2.$
\end{lemma}
The following theorem give fundamental group for some $k,n,M$.
\begin{theorem}\label{rp3}
\[
\pi_{1}(W_{k,n}(\mathbb{R}P^m))=
\begin{cases}
(\mathbb{Z}_2)^n, &k=n\leq m-1,\\
Q_8, &k=n=m=2,\\
Q_8\oplus \mathbb{Z}_2, &k=n=m=3.
\end{cases}
\]
\end{theorem}
\begin{proof}
Let $k=n\leq m-1$. By Lemma \ref{k}, $\pi_1(W_{n,n}(S^m))\cong \pi_1(V_{m+1,n})$, it is convenient to calculate that $\pi_1(V_{m+1,n})=1,n\leq m-1$ via the following sequence of  fiber bundles
\begin{align*}
V_{m,n-1}&\hookrightarrow V_{m+1,n}\rightarrow V_{m+1,1},\\
V_{m-1,n-2}&\hookrightarrow V_{m,n-1}\rightarrow V_{m,1},\\
&\vdots\\
V_{m-j-1,n-j-2}&\hookrightarrow V_{m-j,n-j-1}\rightarrow V_{m-j,1},\\
&\vdots\\
V_{m-(n-3)-1,n-(n-3)-2}&\hookrightarrow V_{m-{n-3},n-(n-3)-1}\rightarrow V_{m-(n-3),1}
\end{align*}
Then $W_{n,n}(S^m)$ is a universal covering space of  $W_{n,n}(\mathbb{R}P^m)$. It is understood that the group of deck transformations of this covering space  is $(\mathbb{Z}_2)^n$, then
\[
\pi_1(W_{k,n}(\mathbb{R}P^m))\cong (\mathbb{Z}_2)^n, k=n\leq m-1.
\]\par
Let $k=n=m=2$, it's easily seen that $W_{2,2}(\mathbb{R}P^2)$ is equal to Fadell's configuration space $F_{2}(\mathbb{R}P^2)$, following J.V.Buskirk \cite{J.V.Buskirk1966}, we obtain that
\[
\pi_1(W_{2,2}(\mathbb{R}P^2))\cong \pi_1(F_{2}(\mathbb{R}P^2))\cong Q_8
\]
where $Q_8=\langle a,b\mid a^4=1,b^2=a^2,b^{-1}ab=a^3\rangle$ is the quaternion group.\par
Let $k=n=m=3$, consider  $W_{3,3}(\mathbb{R}P^3)$ as the orbit space $W_{3,3}(S^3)/(\mathbb{Z}_2)^3$ where the action of $\mathbb{Z}_2$ on $S^3$ is generated by the antipodal map of $S^3$.  Following the result by M.A.Armstrong \cite{Armstrong1968}, we obtain exact sequence
 \[
 1\rightarrow \pi_1(W_{3,3}(S^3))\rightarrow \pi_1(W_{3,3}(\mathbb{R}P^3))\rightarrow  (\mathbb{Z}_2)^3\rightarrow 1.
 \]
By Lemma \ref{k}, $\pi_1(W_{3,3}(S^3))\cong \pi_1(V_{4,3})\cong \mathbb{Z}_2$, then $\pi_1(W_{3,3}(\mathbb{R}P^3))$ is a finite group of order 16 and the quotient group $\pi_1(W_{3,3}(\mathbb{R}P^3))/ \mathbb{Z}_2 $ is isomorphic to  $ (\mathbb{Z}_2)^3$. By the classification of finite groups of order 16, $\pi_1(W_{3,3}(\mathbb{R}P^3))$ maybe  isomorphic to following groups:
\begin{itemize}
  \item (1)$\mathbb{Z}_2\oplus \mathbb{Z}_2\oplus \mathbb{Z}_2\oplus \mathbb{Z}_2$,
  \item (2)$D_8\oplus \mathbb{Z}_2$,
  \item (3)$Q_8\oplus \mathbb{Z}_2$,
  \item (4)$\langle a,b,c\mid a^4=b^2=c^2=1,[b,c]=a^2,[a,b]=[a,c]=1\rangle$.
\end{itemize}
where $D_8=\langle a,b\mid a^4=b^2=1,b^{-1}ab=a^3\rangle$ is the dihedral group.\\
Following F.Rhodes \cite[Theorem 4]{Rhodes1966}, the fundamental group of orbit space $W_{3,3}(\mathbb{R}P^3)$ with the choosen base point $([e_1],[e_2],[e_3]))$ is isomorphic to the reduced fundamental group $\tilde{\sigma}(W_{3,3}(S^3),(e_1,e_2,e_3),(\mathbb{Z}_2)^3)$ of transformation group $(W_{3,3}(S^3), (\mathbb{Z}_2)^3)$. It is easily calculated that $\tilde{\sigma}(W_{3,3}(S^3),(e_1,e_2,e_3),(\mathbb{Z}_2)^3)$  has   twelve  elements of order 4 and three elements  of order 2, then the fundamental group of $W_{3,3}(\mathbb{R}P^3)$ is isomorphic to $Q_8\oplus \mathbb{Z}_2$.
\end{proof}
\begin{remark}
The presentation of  $\pi_1(W_{3,3}(\mathbb{R}P^3),([e_1],[e_2],[e_3]))$ is given as follows.
\begin{align*}
\pi_1(W_{3,3}&(\mathbb{R}P^3),([e_1],[e_2],[e_3]))\\
=&\langle \beta_1,\beta_2,\beta_3\mid \beta_1^4=1,\beta_2^2=\beta_1^2,\beta_2^{-1}\beta_1\beta_2=\beta_1^{-1},\beta_3^2=1,[\beta_1,\beta_3]=[\beta_2,\beta_3]=1\rangle.
\end{align*}
Here $\beta_1,\beta_2,\beta_3 \in \pi_1(W_{3,3}(\mathbb{R}P^3),([e_1],[e_2],[e_3]))$  is defined  by
\begin{align*}
\beta_1(t)&=([\delta_1(t)e_1],[\delta_1(t)e_2],[\delta_1(t)e_3]),\\
\beta_2(t)&=([\delta_2(t)e_1],[\delta_2(t)e_2],[\delta_2(t)e_3]),\\
\beta_3(t)&=([\delta_3(t)e_1],[\delta_3(t)e_2],[\delta_3(t)e_3]).
\end{align*}
where
\[
\delta_1(t)=
\left(
  \begin{array}{ccccc}
    1& 0 &0&0 \\
    0 & 1 & 0&0\\
    0& 0 & \cos \pi t &\sin \pi t \\
   0& 0&-\sin \pi t &\cos \pi t
  \end{array}
\right), 0\leq t\leq 1,
\]
\[
\delta_2(t)=
\left(
  \begin{array}{ccccc}
    1& 0 &0&0 \\
  0  & \cos \pi t &\sin \pi t&0 \\
   0 &-\sin \pi t &\cos \pi t&0\\
    0 & 0 & 0&1\\
  \end{array}
\right), 0\leq t\leq 1,
\]
\[
\delta_3(t)=
\left(
  \begin{array}{ccccc}
  \cos \pi t &\sin \pi t&0  &0\\
   -\sin \pi t &\cos \pi t&0&0\\
   0&0&\cos \pi t &\sin \pi t\\
   0&0&-\sin \pi t &\cos \pi t
  \end{array}
\right), 0\leq t\leq 1.
\]
The proof is as follows. It is understood that the path lifting $\beta_1^2$ starting at $(e_1,e_2,e_3)$ is $\alpha\in \pi_1(W_{3,3}(S^3))$ and $\alpha$ has order $2$  (See the proof of Theorem \ref{m} for  definition of $\alpha$),  then $\beta_1^2$ has order $2$ in the fundamental group of $W_{3,3}(\mathbb{R}P^3)$ via the exact sequence
\[
1\rightarrow \pi_1(W_{3,3}(S^3))\rightarrow \pi_1(W_{3,3}(\mathbb{R}P^3))\rightarrow \oplus_3 \mathbb{Z}_2\rightarrow 1,
\]
thus $\beta_1$ has order $4$ in the fundamental group of $W_{3,3}(\mathbb{R}P^3)$. It's easily seen that $\beta_3$ has order $2$.   By Theorem \ref{rp3}, $\pi_1(W_{3,3}(\mathbb{R}P^3))\cong Q_8\oplus \mathbb{Z}_2$. It's convenient to verify the relations $\beta_2^2=\beta_1^2,\beta_2^{-1}\beta_1\beta_2=\beta_1^{-1},[\beta_1,\beta_3]=[\beta_2,\beta_3]=1$ via the  construction of  the homotopy as in the proof of Theorem \ref{m}, thus we obtain the presentation.
\end{remark}
\begin{theorem}
$\pi_1(W_{4,4}(\mathbb{R}P^4))\cong Q_8\ast_{Z} D_8$
\end{theorem}
\begin{proof}
The proof is similar to that of  Theorem \ref{rp3}. Apply the exact sequence
\[
 1\rightarrow \pi_1(W_{4,4}(S^4))\rightarrow \pi_1(W_{4,4}(\mathbb{R}P^4))\rightarrow \oplus_4 \mathbb{Z}_2\rightarrow 1
\]
 and $\pi_1(W_{4,4}(S^4))\cong \pi_1(V_{5,4})\cong \mathbb{Z}_2$ (by Lemma \ref{k}),  $\pi_1(W_{4,4}(\mathbb{R}P^4),([e_1],[e_2],[e_3],[e_4])) $ is a finite group of order $2^5$ and  $\pi_1(W_{4,4}(\mathbb{R}P^4),([e_1],[e_2],[e_3],[e_4])) /\mathbb{Z}_2\cong (\mathbb{Z}_2)^4$.  Calculate the reduced fundamental group $\tilde{\sigma}(W_{4,4}(S^4),(e_1,e_2,e_3,e_4),(\mathbb{Z}_2)^4)$ of transformation group $(W_{4,4}(S^4), (\mathbb{Z}_2)^4)$, we obtain that there are twenty  elements of order 4 and eleven elements of order 2 in $\tilde{\sigma}(W_{4,4}(S^4),(e_1,e_2,e_3,e_4),(\mathbb{Z}_2)^4)$ which is isomorphic to $\pi_1(W_{4,4}(\mathbb{R}P^4),([e_1],[e_2],[e_3],[e_4]))$. \\
 By the classification of finite group of order $2^5$, $\pi_1(W_{4,4}(\mathbb{R}P^4))$ is isomorphic to an extraspecial 2-group of order $2^5$ which is the central product of $Q_8$ and $D_8$ denoted by $Q_8\ast_{Z} D_8$.
\end{proof}

%
%
%

\bibliographystyle{gtart}
\bibliography{configuration}
%



\end{document}